\documentclass{amsart}

\usepackage{amssymb}
\usepackage{mathrsfs}

\newtheorem{claim}{Claim}
\newtheorem{fact}{Fact}
\newcommand{\Seq}[1]{( #1 )}
\newcommand\N{\mathbf{N}}

\newcommand\op{*}
\newcommand\var{\boldsymbol{1}}

\newcommand{\Th}{{}^{\textrm{th}}}

\title[no idempotent means]{Idempotent means on \\ free binary systems do not exist}



\thanks{
The research represented in this article was funded in part by 
NSF grant DMS--1600635.
}

\author{Justin Tatch Moore}

\address{Department of Mathematics \\ Cornell University\\
Ithaca, NY 14853--4201 \\ USA}

\email{{\tt justin@math.cornell.edu}}
 
\begin{document}

\begin{abstract}
Free binary systems are shown to not admit idempotent means.
This refutes a conjecture of the author.
It is also shown that the extension of Hindman's theorem to 
nonassociative binary systems formulated and conjectured by the author is false.
\end{abstract}

\maketitle

\section{Introduction}

Recall that a \emph{binary system}, or \emph{magma}, is a nonempty set equipped with a binary operation.
If \((S,*)\) is a binary system, then \(*\) can be extended to the set \(M(S)\) of \emph{means} on 
\(S\) by:
\[
\mu * \nu (f) = \int \left( \int f(s * t)\  d \nu (t) \right) d\mu (s).
\]
Here a \emph{mean} on a set \(S\) is a \(\mu \in \ell^\infty (S)^*\) such that \(\mu (\chi_S) = 1\) and \(\mu(f) \geq 0\) whenever \(f \geq 0\).
Such a \(\mu\) is completely determined by its values on the characteristic functions
(via integration)
and in this way we will identify means with finitely additive probability measures.
If \(\mu\) takes only values in \(\{0,1\}\) on characteristic functions,
then we say that \(\mu\) is an \emph{ultrafilter}.

Answering a question of Galvin, Glazer noted that if \(*\) is an associative operation on \(S\),
then Ellis's Lemma \cite{Ellis_lem} implies there is an idempotent ultrafilter on \(S\).
Galvin had already noted that the existence of idempotent ultrafilters on \((\N,+)\) could be used to give a short proof of Hindman's Theorem \cite{Hindman_thm}:
\begin{quote}
\emph{
If \(\N = \bigcup_{i < d} K_i\), then there is an infinite set \(H \subseteq \N\) and an 
\(i < d\) such that whenever \(F\) is a nonempty finite subset of \(H\),
the sum of \(F\) is in \(K_i\).}
\end{quote}
In fact idempotent ultrafilters on semigroups have found extensive applications in Ramsey theory
--- see \cite{alg_betaN} for a more detailed account of both the history and the applications.

A possible extension of this theory to nonassociative binary systems
was considered by the author in \cite{Hind_Ellis_F}.
It was noted there that idempotent ultrafilters do not exist on free binary systems.
On the other hand, it was shown that the existence of an idempotent
mean on any free binary system implies Richard Thomspon's group \(F\) is amenable.
In fact it was demonstrated there that the existence of such idempotent means
implies a version of Hindman's theorem for the free binary system on one generator
which in turn implies that \(F\) is amenable.  

The amenability problem for \(F\) is a long standing problem in group theory first considered by
Richard Thompson \cite{thompson_letter}
but rediscovered and first popularized by Ross Geoghegan;
see \cite[p. 549]{comb_grp_top}. 
It is arguably the most notorious problem concerning the amenability of a specific group.
I previously (and incorrectly) claimed to have proved the existence of an idempotent mean on a
free binary system \cite{amen_F}.
In this article I will prove that free binary systems does not support idempotent means,
refuting Conjecture 1.4 of \cite{Hind_Ellis_F}.
I will also refute Conjecture 1.3 of \cite{Hind_Ellis_F}, which was an extension of Hindman's theorem to nonassociative binary systems.
The question of \(F\)'s amenability remains open.

While this article is self contained, the reader will find additional motivation in
\cite{Hind_Ellis_F} (and \cite{amen_F}).
Throughout this paper, counting will start with \(0\).
The variables \(i,j,k,m,n,p\) will always implicitly
be taken to range over the nonnegative integers.
For instance, the \(n\)-tuple \(\Seq{a_0,\ldots,a_{n-1}}\)
will be denoted \(\Seq{a_k \mid k < n}\).

\section{Free binary systems do not support idempotent measures}

Let \((S,*)\) be a free binary system generated by \(I\), fixed for the remainder of the section.
Notice that the binary operation \(*\) is an injection from \(S \times S \to S \setminus I\);
this is in fact equivalent to the freeness of \((S,*)\).
Define \(\#:(S,*) \to (\N,+)\) to be the homomorphism which maps every element of \(I\) to \(1\).
Thus \(\#(s)\) is the size of the (unique) nonassociative product of generators used to produce \(s\).
In particular, if \(s = a * b\), then \(\#(a), \#(b) < \#(s) = \#(a) + \#(b)\).
Set \(S_n :=\{s \in S \mid \#(s) = n\}\).

Define membership to sets \(Z \subseteq S\) and \(T_p \subseteq S\) recursively on \(\#(\cdot)\)
as follows:
\begin{itemize}

\item \(T_0 = S\) and \(T_{p+1} = (S \setminus Z) * T_p\);

\item \(s \in Z\) if and only if \(s = a * b\) where \(b \in T_{\#(a)}\).

\end{itemize}
Observe that \(Z = \bigcup_p S_p * T_p\).

Recall that if \(\mu\) and \(\nu\) are means on \(S\),
then
\[
\mu * \nu (f) = \int \left( \int f(s * t)\  d \nu (t) \right) d\mu (s)
\]
defines a mean on \(S\).
We note the following facts.

\begin{fact} \label{prod_fact}
If \(\mu\) and \(\nu\) are means on \(S\) and \(A,B \subseteq S\),
then \(\mu * \nu(A *B) = \mu(A) \nu(B)\).
\end{fact}

\begin{fact} \label{avg_fact}
If \(X \subseteq S\) and for some \(c\) and \(\mu\)-a.e. \(s \in S\),
\(
\nu(\{t \in S \mid s * t \in X\}) = c
\)
then \(\mu * \nu(X) = c\).
\end{fact}

Now suppose for contradiction that \(\mu\) is an idempotent mean on \(S\) and set \(r := \mu(E)\).
By Fact \ref{prod_fact},
\[
\mu(S_1) = \mu (I) = \mu(S \setminus (S*S)) = 1 - \mu(S)\mu(S) = 0.
\]
Fact \ref{prod_fact} inductively implies that \(\mu(T_n) = (1-r)^n\) and that
\[
\mu(S_n) = \mu (\bigcup_{i+j=n} S_i * S_j) = \sum_{i+j=n} \mu(S_i)\mu(S_j) = 0
\]
for all \(n\).
Moreover \(\mu(S_n * S) = 0\) for all \(n\).

If \(r > 0\), then let \(n\) be sufficiently large that \((1-r)^n < r\).
Observe that
\[
Z = \bigcup_{k < n} S_k * T_k \cup \bigcup_{k = n}^\infty S_k * T_k \subseteq 
\bigcup_{k < n} S_k * T_k \cup \bigcup_{k = n}^\infty S_k * T_n.
\]
Since \(\mu ( \bigcup_{k < n} S_k * T_k ) = 0\) and since
\(\mu (\bigcup_{k = n}^\infty S_k * T_n) = \mu(T_n) = (1-r)^n\),
we have \(r = \mu(Z) \leq (1-r)^n < r\), contrary to our choice of \(n\).

If \(r = 0\), then for every \(s \in S\),
\[
\mu(\{t \in S \mid s * t \in Z\} ) = \mu( T_{\#(s)} ) = (1-r)^{\#(s)} = 1.
\]
However by Fact \ref{avg_fact}, \(\mu (Z) = \mu*\mu(Z) = 1 \ne 0 = \mu (Z)\), which is also a contradiction.
This refutes Conjecture 1.4 of \cite{Hind_Ellis_F} which asserts that whenever \((S,*)\) is a binary system,
every compact convex subsystem of \(M(S)\) contains an idempotent.

\section[nonassociative]{The nonassociative analog of Hindman's Theorem is false}

In this section we will refute Conjecture 1.3 of \cite{Hind_Ellis_F}, which was suggested
as a nonassociative analog of Hindman's Theorem.
For the duration of the section, we will let
\((S,*)\) denote the free binary system on one generator \(\var\).
The sets \(S_n\), \(T_n\), and \(Z\) are defined as in the previous section.
Note that in the present context \(S_n\) is finite for each \(n\).
If \(\mu \in \bigcup_p M(S_p)\), we will write \(\#(\mu)\) to denote the unique
\(p\) such that \(\mu \in M(S_p)\).
Conjecture 1.3 of \cite{Hind_Ellis_F} can now be stated as follows:
\begin{quote}
\emph{If \(c:S \to [0,1]\)
and \(\epsilon > 0\), then there is an \(r \in [0,1]\) and an
increasing sequence \(\Seq{\mu_i \mid i < \infty}\) of elements of
\(\bigcup_p M(S_p)\) such that whenever \(s\) is in \(S_n\) and
\(\Seq{i_k \mid k < n}\) is admissible for \(s\):} 
\[
|c(s\Seq{\mu_{i_k} \mid k < n}) - r | < \epsilon
\]
\end{quote}
Here \(s\Seq{\mu_{i_k} \mid k < n}\) is the result of taking the unique term used to generate \(s\)
from \(\var\) and \(\op\), replacing the \(k\Th\) occurrence of \(\var\) with \(\mu_{i_k}\), and evaluating
the resulting expression in \(M(S)\).
Notice that if \(\#(\mu_{i_k}) = p_k\), then \(\#(s\Seq{\mu_{i_k} \mid k < n}) = \sum_{k < n} p_k\).

We will work the following equivalent recursive definition of \emph{admissible}:
\begin{itemize}

\item \(\var\) is admissible for any sequence of positive integers of length \(1\);

\item if \(s \in S_m\) and \(t \in S_n\), then
\(\Seq{i_k \mid k < m+n}\) is admissible for \(s \op t\) if and only if
\(\Seq{i_k \mid k < m}\) is admissible for \(s\) and 
\(\Seq{i_k - m \mid m \leq k < n}\) is admissible for \(t\).

\end{itemize}
Thus \(\Seq{i_k \mid k < m}\) is admissible for any element of $S_m$
provided that $m \leq i_0$.

Returning to Conjecture 1.3 of \cite{Hind_Ellis_F}, we will show that the conclusion of the conjecture fails
if \(c\) is defined by \(c(\mu) = \mu(Z)\) and \(0 < \epsilon < 1/2\).
Toward this end, let \(\Seq{\mu_i \mid i < \infty}\) be given and let \(r\) be any accumulation point
of the sequence \(\Seq{\mu_i (Z) \mid i < \infty}\).
We will be finished once we prove the following three claims.

\begin{claim} \label{near0}
For every \(m\) and \(\epsilon > 0\), there is an \(s \in S\) and 
\(\Seq{i_k \mid k < n}\) such that:
\begin{enumerate}

\item \(\Seq{i_k - m \mid k < n}\) is admissible for \(s\);

\item \(s\Seq{\mu_{i_k} \mid k < n}(Z) < \epsilon\).
 
\end{enumerate}
\end{claim}

\begin{proof}
Let \(m\) and \(\epsilon > 0\) be given.
If \(r=0\), then we can take \(s = \var\) and choose \(i_0 > m\) so that
\(\mu_{i_0}(Z) < \epsilon\).
Therefore suppose that \(r > 0\) and let \(l\) be sufficiently large that \((1-r)^l < \epsilon\).
If \(k < l\), define \(i_k = m + k + 1\) and let \(u := \var \op (\cdots \var *(\var * \var) \cdots)\) be the right associated product
of \(l\) many \(\var\)'s.
Thus \(\Seq{i_k - m \mid k < l}\) is admissible for \(u\).
Set \(p:= \#(u\Seq{\mu_{i_k} \mid k < l})\), \(n:=l + p + 1\), and
let \(v\) be the right associated product of \(p+1\) many \(\var\)'s.
Since \(p \geq l\), by choice of \(r\) it is possible to pick an increasing sequence
\(\Seq{i_k \mid l \leq k < n}\) of indices such that
\(i_l > m+l\) and:
\[
\prod_{k=l}^{l + p + 1} (1 -\mu_{i_k}(Z)) < \epsilon
\]
This implies in particular that \(v\Seq{\mu_{i_k} \mid l \leq k < n}(T_p) < \epsilon\).
Set \(s := u \op v\) and
observe that \(\Seq{i_k - m- l \mid l \leq k < n}\) is admissible for \(v\) and thus
\(\Seq{i_k - m \mid k < n}\) is admissible for \(s\).
Since \(u\Seq{\mu_{i_k} \mid k < l}(S_p) = 1\) and since
\(Z \cap (S_p * S) = S_p * T_p\), we have that
\(s\Seq{\mu_{i_k} \mid k < n}(Z) < \epsilon\) as desired.
\end{proof}

\begin{claim} \label{T_claim}
For every \(m\) and \(p\) and every \(\epsilon > 0\),
there is an \(s \in S\) and \(\Seq{i_k \mid k < n}\) such that:
\begin{enumerate}

\item \(\Seq{i_k - m \mid k < n}\) is admissible for \(s\);

\item \(s\Seq{\mu_{i_k}  \mid k < n}(T_p) > 1-\epsilon\).
 
\end{enumerate}
\end{claim}

\begin{proof}
The proof is by induction on \(p\).
The base case is trivial since \(T_0 = S\).
Suppose the claim holds for a given \(p\) and
let \(m\) and \(\epsilon > 0\) be given.
Fix \(\delta > 0\) such that \((1-\delta)^2 > 1-\epsilon\).
By Claim \ref{near0}, there are \(u \in S\) and \(\Seq{i_k \mid k < l}\) such that
\[
u\Seq{\mu_{i_k} \mid k < l}(Z) < \delta
\]
and \(\Seq{i_k - m  \mid k < l}\) is admissible for \(u\).
By our inductive hypothesis, there exist \(v\) and \(\Seq{i_k \mid l \leq k < n}\) such
that
\[
v\Seq{\mu_{i_k} \mid l \leq k < n}(T_p) > 1-\delta
\]
and
\(\Seq{i_k - m - l  \mid l \leq k < n}\) is admissible for \(v\).
It follows that \(s := u \op v\) and \(\Seq{i_k \mid k < n}\) satisfy the conclusion of
the claim.
\end{proof}

\begin{claim} \label{near1}
For every \(\epsilon > 0\), there is an \(s \in S\) and \(\Seq{i_k \mid k < n}\) which is admissible
for \(s\) such that \(s\Seq{\mu_{i_k} \mid k < n}(Z) > 1-\epsilon\).
\end{claim}

\begin{proof}
Setting \(p :=\#(\mu_0)\), 
by Claim \ref{T_claim} there is a sequence \(\Seq{i_k \mid 0 < k < n}\) and a \(t \in S\) such that
\(\Seq{i_k -1 \mid 0 < k < n}\) is admissible for \(t\) and
\[
t\Seq{\mu_{i_k} \mid 1 \leq k < n}(T_p) > 1-\epsilon.
\]
Setting \(i_0:=0\), \(s := \var \op t\) and \(\Seq{i_k \mid k < n}\) satisfy the conclusion of the claim.
\end{proof}

The desired contradiction to Conjecture 1.3 of \cite{Hind_Ellis_F} now follows from
Claims \ref{near0} and \ref{near1} with \(m=0\)
by noting that for any \(0 \leq r \leq 1\) and \(0 < \epsilon < 1/2\),
it is not possible for both \(0\) and \(1\) to be
in the interval \([r-\epsilon,r+\epsilon]\).

\def\Dbar{\leavevmode\lower.6ex\hbox to 0pt{\hskip-.23ex \accent"16\hss}D}

\end{document}